\documentclass[11pt,oneside]{amsart}
\usepackage{amsthm}
\usepackage{amsmath}
\usepackage{amssymb}
\usepackage{graphicx}
\usepackage{color}
\usepackage{url}
\makeatletter\def\url@leostyle{%
  \@ifundefined{selectfont}{\def\UrlFont{\sf}}{\def\UrlFont{\scriptsize\ttfamily}}} \makeatother

\date{}

\newtheorem{proposition}{Proposition}
\newtheorem{theorem}{Theorem}

\theoremstyle{definition}
\newtheorem{definition}[theorem]{Definition}
\newtheorem{remark}[theorem]{Remark}

\title{Parameter estimations for SPDEs with multiplicative fractional noise }
\author{Igor Cialenco}
\address{
Department of Applied Mathematics,
Illinois Institute of Technology \\
10 West 32nd Str, Bld E1, Room 208, \\
Chicago, IL 60616, USA }
\email{igor@math.iit.edu, http://math.iit.edu/$\sim$igor}
\thanks{Research supported in part by the the NSF grant DMS-0908099.}

\date{\today}

\begin{document}

\begin{abstract}

We study parameter estimation problem for diagonalizable  parabolic stochastic partial differential equations
driven by a multiplicative fractional noise with any Hurst parameter $H\in(0,1)$.
Two classes of estimates are investigated: traditional maximum likelihood type estimates,
and a new class called closed-form exact estimates. Finally several examples are discussed, including
statistical inference for stochastic heat equation driven by a fractional Brownian motion.

\vskip 0.3 true cm
\noindent {\bf  AMS 2010:} Primary 60H15; Secondary 62F12, 60G22  \vskip 0.2 true cm
\noindent {\bf  Keywords:} Asymptotic normality, parameter estimation, stochastic PDE, multiplicative noise, singular models
\end{abstract}
\maketitle

\section{Introduction}
Parameter estimation problem for stochastic partial differential equation
has been of great interest  in the past decade, and besides being a challenging theoretical problem,
it finds its roots and motivations from various applied problems:
fluid dynamics \cite{Frankignoul1985, PiterbargRozovskii1996},  biology \cite{Dawson1980, De1987},
finance \cite{AiharaBagchi2005, AiharaBagchi2006, Cont2005}, meteorology \cite{DuanChen2010} etc.
At general level the problem is to find or estimate the model parameter $\vartheta$ (could be a vector)
based on observations of the  underlying process $u_\vartheta$ which is assumed to be a solution of
a stochastic evolution equation in finite or infinite dimensional space.  We will follow traditional continuous time
approach and assume that the solution $u_\vartheta(t)$ is observed continuously in time $t\in[0,T]$.
From statistical point of view, we suppose  that there exists a family of probability measures $\mathbf{P}_\vartheta$ that depends on
parameter $\vartheta\in\Theta\subset\mathbb{R}^n$, and each $\mathbf{P}_\vartheta$ is the distribution of  a random element.
Assuming that a realization of one random element corresponds to a particular value $\vartheta_0$, the goal is to estimate
this parameter from given observations. One approach is to select parameter $\vartheta$ that {\it most likely} produces the observations.
This method assumes that the problem is regular or absolutely continuous, which means that
there exists a reference probability measure $\mathbf{Q}$ such that all measures $\mathbf{P}_\vartheta, \ \vartheta\in\Theta$, are
absolutely continuous with respect to $\mathbf{Q}$. Then Radon-Nikodym derivative $d\mathbf{P}_\vartheta/d\mathbf{Q}$, also called the likelihood ratio, exists, and
{\it the Maximum Likelihood Estimator} (MLE) $\widehat{\vartheta}$ of the parameter of interest is computed by maximizing the likelihood ratio with respect to $\vartheta$.
Usually $\widehat{\vartheta}\neq \vartheta$ and the problem is to study the convergence of MLE to the true parameter as more information arrives (for example
as time passes or by decreasing the amplitude of noise).
If the measures $\mathbf{P}_\vartheta$ are singular for different parameters $\vartheta$, then the model
is called singular, and usually the parameter can be found exactly, at least theoretically.
While all regular models are to some extend  the same, each singular model requires individual approach. For example, estimating the drift
coefficient for finite-dimensional stochastic differential equations is typically a regular problem, and the parameter can be estimated by means of MLEs,
while estimating the diffusion (volatility) coefficient is a singular problem and one can find the diffusion coefficient
exactly through quadratic variation of the underlying process. For some finite-dimensional systems,
estimating the ``drift coefficient''  is also a singular problem, and as shown in Khasminskii et al
\cite{KhasminskiiKrylovMoshniuk} the estimators have nothing to do with MLEs.
Generally speaking statistical inference for finite-dimensional diffusions has been studied widely,
and there are established necessary and sufficient conditions for absolute continuity of corresponding
measures (see, for example \cite{LiptserShiryayev}, \cite{KutoyantsBook2004} and references therein). Some of these results have been extended to infinite dimensional systems in particular to
parabolic Stochastic Partial Differential Equations (SPDE). It turns out that in many cases the estimation of drift coefficient for SPDEs is a singular problem,
and as general theory suggests one can find the parameter ``exactly''.  One of the first fundamental result in this area that
explorers this singularity is due to Huebner, Rozovskii, and Khasminskii \cite{HubnerRozovskiiKhasminskii}.  The idea is to approximate
the original singular problem by a sequence of regular problems for which MLEs exist.
The approximation was done by considering Galerkin-type of projections of the solution on a finite-dimensional space where the estimation
problem becomes regular, and it was proved that as dimension of the projection increases the corresponding MLE will converge to the true parameter.
In \cite{HuebnerLototskyRozovskii97, HuebnerRozovskii,  LototskyRozovskii1999, LototskyRozovskii2000}, the problem was extended to a
general parabolic SPDE driven by additive noise and the convergence of the estimators was given in terms of the
order of the corresponding differential operators.  For recent developments and  other types of inference
problems in SPDEs see a survey paper by Lototsky \cite{Lototsky2009Survey}
and references therein. Statistical inference for SPDEs driven by multiplicative noise is a more challenging problem.
First and only attempt to study equations with multiplicative noise is given in \cite{IgorSergey2007},
by considering Wiener (not fractional) type noise  without spatial correlation structure. Besides MLE type estimates,
a  completely new class of {\it exact estimates}  were found due essentially to the very singular nature of the problem.

The aim of this note is to study parameter estimation problem for stochastic parabolic equations driven
by a {\it multiplicative fractional noise} with following dynamics
\begin{equation}\label{eq:mainIntro}
u(t) = u(0) + \int_0^t(\mathcal{A}_0+\theta\mathcal{A}_1)u(s)ds + \int_0^t\mathcal{M}u(s)dW^H(s),
\end{equation}
where $\mathcal{A}_0, \mathcal{A}_1$ and $\mathcal{M}$ are some known linear operators,
$W^H$ is a fractional Brownian motion with a Hurst parameter $H\in(0,1)$, and $\theta$ is a real
parameter belonging to a bounded set $\Theta\subset\mathbb{R}$.
For now, assume that the stochastic integral with respect to
fractional Brownian Motion $W^H$ is well-defined, while the exact meaning will
be specified in Section \ref{subsection:Existance}.
The main goal is to estimate the parameter $\theta$ based on the observations of the underlying process $u(t), \ t\in[0,T]$.
Similar problem for SPDEs  driven by additive space-time fractional
noise was investigated in \cite{IgorSergeyJan2008, PrakasaRao2004, MaslowskiPospisil2008}.
Estimation of drift coefficient for finite-dimensional fractional Ornstein-Uhlenbeck and similar
processes has been investigated by Tudor and Viens \cite{ViensTudor2007} for $H\in(0,1)$,
Kleptsyna and Le Breton \cite{KleptsynaBreton2002} for $H\in[1/2,1)$, by developing Girsanov type theorems and finding
MLEs. Berzin and Leon \cite{BerzinLeon2008} estimate simultaneously both drift and diffusion coefficients.
Least square estimates for drift coefficients were established by Hu and Nualart \cite{NualartHu2009}, and MLE type estimates for
discretely observed process by Hu, Weilin and Weiguo \cite{HuWeilinWeiguo2009}. For a general theory, including Girsanov Theorem and some results on statistical inference,
for finite dimensional diffusions driven by fractional noise see also the monograph by Mishura \cite{MishuraBook2008}.

In this paper we continue to explore the impact of the noise in  infinite-dimensional evolution equations
and its implications on statistical inference. Besides its theoretical roots, this problem is also motivated by
increasing demand in modeling various phenomena by SPDEs driven by fractional noise \cite{ DuanChen2010, Duan2009}.
We assume that the solution  of \eqref{eq:mainIntro}  is observed at every
$t\in[0,T]$, and hence each Fourier coefficient $u_k(t)=(u_k(t), h_k)_H$ is observable for every $t\in[0,T]$, where
$H$ is a Hilbert space in which the solution leaves and $h_k, k\geq 1,$ is a CONS in $H$. All results are stated in terms of Fourier coefficients
$u_k$.
 In the first part of Section \ref{section:PreliminaryResults} we
set up the problem and establish the existence and uniqueness of the solution of the corresponding SPDE.
In Subsection \ref{subsection:estimationGeometricalFBM} we introduce the main notations and
find the MLE for fractional Geometrical Brownian Motion
(which is not covered explicitly in any other sources, at our best knowledge). In Section \ref{section:MLEforSPDE}
we study the estimates of drift coefficient $\theta$ of equation \eqref{eq:mainIntro} based on MLE of the corresponding Fourier coefficients.
We establish sufficient conditions on operators $\mathcal{A}_1, \ \mathcal{A}_1$ and $\mathcal{M}$, that guarantee efficiency and
asymptotic normality of the estimates and some of their versions.  Section \ref{section:ExactEstimates} is
dedicated to investigation of  a new type of estimates called {\it closed-form exact} estimates, similar to those studied in \cite{IgorSergey2007}.
We show that $\theta$ can be found {\tt exactly} by knowing just several (usually two) Fourier coefficients.
Moreover, by the same technics we found an exact estimate of the Hurst parameter $H$ too, in both regimes, $\theta$ known and unknown.
Of course there are many other methods of finding the Hurst parameter, but it is out of scope of this publication to apply them to our equation.
Some of the results follow from simple algebraic evaluations, but the very existence of such estimates is amazing and gives a better understanding of
the nature of the problem's singularity.
Also, we want to mention that, despite of memory property of the fractional Brownian Motion which is spilled over the solution too, the exact estimates
are based only on observations at time zero and some future time $T$. In contrast, the MLEs require observation of the whole trajectory $u(t), \ t\in[0,T]$.
We conclude the paper with two examples which are of interest along: stochastic heat equation with parameter $\theta$ next to Laplace operator,
and a general second order parabolic SPDE with $\theta$ next to a lower order operator.

While we assume  that data is sampled continuously in time, in practice usually this is not the case.
For the MLEs derived in Section \ref{section:MLEforSPDE} the problem is reduced to approximate
some integrals of a deterministic function with respect
to the solution $u$ and eventually to the fractional Brownian motion.
However, the Exact Estimates from Section \ref{section:ExactEstimates} depend only on the values of the solution  at initial
time $t=0$ and some future time $t=T$, and thus do not depend on how the solution is observed in time.

\section{Preliminary results}\label{section:PreliminaryResults}
\subsection{The equation and existence of the solution}\label{subsection:Existance}
Let $\mathbf{H}$ be a separable Hilbert space with the inner product $(\cdot,
\cdot)_0$ and the corresponding norm $\|\cdot\|_0$.
Let $\Lambda$ be  a
densely-defined
linear operator on $\mathbf{H}$ with the following property:
there exists a positive number $c$ such that
 $\| \Lambda u\|_0\geq c\|u\|_0$ for
every $u$ from the domain of $\Lambda$.
Then the operator powers  $\Lambda^\gamma, \ \gamma\in\mathbb R,$ are
 well defined and generate the spaces $\mathbf{H}^\gamma$:
 for $\gamma>0$, $\mathbf{H}^\gamma$ is the domain of
  $\Lambda^\gamma$; $\mathbf{H}^0=\mathbf{H}$;
 for $\gamma <0$,  $\mathbf{H}^\gamma$ is the completion
 of $\mathbf{H}$ with respect to the norm
 $\| \cdot \|_\gamma := \|\Lambda \cdot\|_0$ (see for instance
 Krein at al. \cite{KreinPetuninSemenov}).
 By construction, the collection of spaces
 $\{ \mathbf{H}^\gamma,\ \gamma\in \mathbb R\}$
  has the following properties:
 \begin{itemize}
 \item $\Lambda^{\gamma}(\mathbf{H}^r) = \mathbf{H}^{r-\gamma}$ for every
$\gamma,r\in\mathbb{R}$;
 \item For $\gamma_1<\gamma_2$ the space $\mathbf{H}^{\gamma_2}$ is densely and
 continuously embedded into $\mathbf{H}^{\gamma_1}$: $\mathbf{H}^{\gamma_2}
 \subset \mathbf{H}^{\gamma_1}$ and there exists a positive
 number $c_{12}$ such that  $\|u\|_{\gamma_1}\leq c_{12} \|u\|_{\gamma_2}$
 for all $u\in \mathbf{H}^{\gamma_2}$ ;
\item  for every
$\gamma\in\mathbb R$ and $m>0$, the space $\mathbf{H}^{\gamma -m}$ is the
dual of $\mathbf{H}^{\gamma+m}$ relative to the inner product in
$\mathbf{H}^{\gamma}$, with duality $\langle\cdot,\cdot\rangle_{\gamma,m}$ given by
$$
\langle u_1, u_2 \rangle_{\gamma,m} = (\Lambda^{\gamma -m}u_1,
\Lambda^{\gamma+m}u_2)_0, \ {\rm where\ }
u_1\in\mathbf{H}^{\gamma-m},\ u_2\in\mathbf{H}^{\gamma +m}.
$$
\end{itemize}

Let $(\Omega, \mathcal{F}, \{\mathcal{F}_t\}, \mathbb{P})$ be a stochastic basis
with usual assumptions.

\begin{definition} \label{def-cfBM}
A fractional Brownian motion with a Hurst parameter
$H\in (0,1)$ is a Gaussian process $W^H$ with zero mean and
covariance
$$
\mathbb{E} W^H(t)W^H(s)=\frac{1}{2}(t^{2H}+s^{2H}-|t-s|^{2H}), \quad t,s\geq 0.
$$
\end{definition}

Consider the
following evolution equation
\begin{equation}\label{eq:main}
\begin{cases}
du(t) = [(\mathcal{A}_0 + \theta\mathcal{A}_1)u(t) + f(t)]dt + (\mathcal{M}u(t) + g(t))dW^H(t) , \ \ 0<t<T, \\
u(0) = u_0\,  ,
\end{cases}
\end{equation}
where $\mathcal{A}_0, \mathcal{A}_1, \mathcal{M}$ are linear operators in $\mathbf{H}$, $f$ and $g_k$ are
adapted $\mathbf{H}$-valued processes, $u_0\in\mathbf{H}$, $W^H$ is a fractional Brownian Motion with Hurst parameter
$H\in(0,1)$, and $\theta$ is a scalar parameter bellowing to an open set $\Theta\subset \mathbb{R}$.

\begin{definition}\label{def-diag}
 Equation \eqref{eq:main} is called diagonalizable if
the operators $\mathcal{A}_0, \ \mathcal{A}_1$  and $\mathcal{M}$ have point
spectrum and a common system of eigenfunctions
$\{h_j,\ j\geq 1\}.$
\end{definition}

Denote by $\rho_k$, $\nu_k$, and $\mu_k$ the
eigenvalues of the operators $\mathcal{A}_0$, $\mathcal{A}_1$,
and $\mathcal{M}$:
\begin{equation}
\mathcal{A}_0h_k=\rho_kh_k,\ \ \mathcal{A}_1h_k=\nu_kh_k,\ \
\mathcal{M}h_k=\mu_kh_k,  \ k\geq 1,
\end{equation}
and also denote by $\alpha_k(\theta):= \rho_k+\theta\nu_k, k\geq 1$,  the eigenvalues of
operator $\mathcal{A}_0+\theta\mathcal{A}_1$.
Without loss of generality we assume that the operator $\Lambda$ has the same eigenfunctions
as operators $\mathcal{A}_0, \  \mathcal{A}_1, \ \mathcal{M}$: $\Lambda h_k = \lambda_kh_k, \ k\geq 1$.
\begin{definition}
The equation \eqref{eq:main} is called {\it parabolic} in the triple
$(\mathbf{H}^{\gamma+m}, \mathbf{H}^\gamma, \mathbf{H}^{\gamma-m})$,
for some positive $m$ and real $\gamma$, if there exists
positive real numbers $\delta, \ C_1$ and a real number $C_2$ such that, for all $k\geq 1$ and
all $\theta\in\Theta$,
\begin{align}
& \lambda_k^{-2m}|\rho_k + \theta\nu_k|\leq C_1; \label{eq:parabolic1} \\
& 2(\rho_k+\theta\nu_k) + \mu_k^2 + \delta\lambda_k^{2m} \leq C_2. \label{eq:parabolic2}
\end{align}
\end{definition}
This definition is equivalent to the classical definition of parabolic equations,
but written in terms of eigenvalues of corresponding operators.
\begin{theorem}\label{th:existance}
Assume that equation \eqref{eq:main} is diagonalizable and parabolic in the triple
$(\mathbf{H}^{\gamma+m}, \mathbf{H}^\gamma, \mathbf{H}^{\gamma-m})$, the initial conditions
$u_0$ is deterministic and belongs to $\mathbf{H}^\gamma$, the process
$f=f(t)$ is $\mathcal{F}_t$-adapted with values in $\mathbf{H}^{\gamma-m}$, and
$\mathbb{E}\int_0^T\|f(t)\|_{\gamma-m}^2dt < \infty$, the process
$g = g(t)$ is $\mathcal{F}_t$-adapted with values in $\mathbf{H}^{\gamma}$ and
$\mathbb{E}\int_0^T\|g(t)\|^2_\gamma \,dt< \infty$.
Then the process $u$ defined by
\begin{equation}\label{eq:solution}
u(t)=\sum\limits_{k\geq 1} u_k(t)h_k,
\end{equation}
where
\begin{align}
& u_k(t)=u_k(0)\exp\left([\alpha_k(\theta)+f_k(t)]t -
\frac{1}{2}[\mu_k+g_k(t)]^2t^{2H} + [\mu_k+g_k(t)]W^H(t)\right), \label{eq:solution2}\\
& f(t)=\sum\limits_{k\geq 1}f_k(t)h_k, \ g(t)= \sum\limits_{k\geq 1} g_k(t)h_k, \nonumber
\end{align}
is an $\mathbf{H}^\gamma$-valued stochastic process.
\end{theorem}
\begin{proof}
Since $W^H(t)$ is a Gaussian random variable with zero mean and variance $t^{2H}$, we have
$$
\mathbb{E}|u_k(t)|^2 = u_k^2(0) \exp\left( 2(\alpha_k(\theta)+f_k(t)) t + (\mu_k+g_k(t))^2t^{2H}  \right).
$$
Hence,
$$
\mathbb{E}\|u(t)\|_\gamma^2 = \sum\limits_{k\geq 1} \lambda_k^{2\gamma} |u_k(t)|^2 \leq
C \sum\limits_{k\geq 1} \exp\left( 2\alpha_k(\theta)t + \mu_k^2t^{2H} \right).
$$
By parabolicity condition \eqref{eq:parabolic2}, the last series converges uniformly in $t$, and the
theorem follows.
\end{proof}

The functions $u_k$ formally represent the Fourier coefficients of the solution of equation
\eqref{eq:main} with respect to the basis $\{h_k\}_{k\geq 1}$ and the uniqueness of $u$ follows.
Since the equation is diagonalizable, naturally we conclude that formally $u_k$ has the following dynamics
\begin{equation}\label{eq:FourierCoeff}
du_k(t) = (\theta \nu_k +\rho_k)u_k(t)dt + \mu_k u_k(t)dW^H(t), \quad k\geq 1, t\geq 0.
\end{equation}
Specifying the stochastic integration in \eqref{eq:main} is equivalent to specifying in what sense we understand the integration
with respect to fractional Brownian Motion for the Fourier coefficients \eqref{eq:FourierCoeff}.
Consequently, since the equation has constant coefficients,
specifying the solution of \eqref{eq:FourierCoeff} is the same as to stipulate the sense of stochastic integration in \eqref{eq:FourierCoeff}.
If the integration is understood in Wick sense then $u_k, \ k\geq 1,$ defined in \eqref{eq:solution2} is the unique solution of
equation \eqref{eq:FourierCoeff} for all $H\in(0,1)$ (see for instance \cite{BiaginiHuOksendalBookfBM2008}, Theorem 6.3.1).
All results stated here are easily transferable to any other form of integration, by caring out
the relationship between different form of integration
and consequently adjusting the form of the solution of equation \eqref{eq:FourierCoeff} (for comparison of various form of integration with respect to
fBM see \cite{BiaginiHuOksendalBookfBM2008}, Chapter 6).
Our choice was just to have a unified theory and same formulas for all $H\in(0,1)$.

\begin{definition}\label{def:solution}
The process $u$ constructed in Theorem \ref{th:existance} is called the solution of equation
\eqref{eq:main}.
\end{definition}

It should be mentioned that the above result, with some obvious adjustments,
 also holds true for diagonalizable equations driven
by several independent fractional Brownian Motions, even with different Hurst parameters.

\subsection{Parameter estimation for geometrical fractional Brownian motion}\label{subsection:estimationGeometricalFBM}
In this section we will present some auxiliary results about parameter estimation for one dimensional
diffusion processes driven by multiplicative fractional noise. For similar results for equations with additive noise
see for instance Kleptsyna and Le Breton \cite{KleptsynaBreton2002}, Tudor and Viens \cite{ViensTudor2007}, or
Mishura \cite{MishuraBook2008}, Chapter~6. The results essentially follow from Girsanov type theorem for
diffusions driven by fractional Brownian motion.

Let $\Gamma$ and $\mathrm{B}$ denote the Euler Gamma-functions.
Following Mishura \cite{MishuraBook2008} we introduce the following notations
\begin{align}
C_H & = \left( \frac{\Gamma(3-2H)}{2H\Gamma(\frac32 - H)^3 \Gamma(\frac12 + H)} \right)^\frac12 , \nonumber \\
l_H(t,s) & = C_H s^{\frac12 - H} (t-s)^{\frac12 - H}\mathbb{I}_{0<s<t} , \label{eq:notations} \\
M_t^H &:= \int_0^t l_H(t,s) dW_s^H,  \, , \nonumber
\end{align}
where $H\in(0,1)$, and the integration with respect to fractional Brownian Motion is understood in Wiener sense (for more details see \cite{MishuraBook2008}, Chapter 1).
The process $M_t^H$ is a martingale, also called the fundamental martingale associated with fractional
Brownian motion $W^H_t$ (see for instance \cite{NorrosValkeilaVirtamo1999} or \cite{MishuraBook2008}, Theorem 1.8.1). $M_t^H$ has quadratic characteristic
$\langle M^H\rangle_t = %t^{1-2\alpha} =
t^{2-2H}$, and by L\'evy theorem, there exists a Wiener process $\{B_t, t\geq 0 \}$ on the same probability space
such that
$$
M_t^H = (2 - 2H)^\frac12 \int_0^t s^{\frac12 -H}dB_s  .
$$
Moreover, $\sigma(W_s^H, 0\leq s\leq t) = \sigma(B_s, 0\leq s\leq t)$.

Let us consider the stochastic process of the form
$$
X_t = X_0\exp\left(\theta t -\frac{1}{2}\sigma^2 t^{2H} + \sigma W^H(t)\right), \quad t\geq 1\, ,
$$
which can be called the Geometric Fractional Brownian Motion, and as mentioned in the previous subsection
it is the unique solution of the stochastic equation
$$
dX_t = \theta X_tdt + \sigma X_tdW^H_t\, , \ X_0=x_0, \ t\in[0,T].
$$

Let $Y_t := \ln X_t/X_0 = \theta t - \frac{\sigma^2 t^{2H}}{2} + \sigma W_t^H$, and
consider the process $\widetilde{Y}_t : = \int_0^t l_H(t,s)dY_s$. Note that observing one path of the process
$\{ Y_s,\ 0\leq s\leq t\}$ implies that the one path of process $\{ \widetilde{Y}_s, \ 0\leq s\leq t\}$ is observable too.
By \eqref{eq:notations} we have
\begin{equation}
\widetilde{Y}_t  = \sigma M_t^H + \theta b_1 t^{2-2H} - \sigma^2 Hb_2 t,  \quad t>0\, ,
\end{equation}
where $b_1= C_H \mathrm{B}(3/2 - H, 3/2-H), \
b_2 = C_H\mathrm{B}(1/2+H, 3/2 - H)$.

For a fixed parameter $\theta\in\Theta$, let us denote by $\mathbb{P}_\theta$ the distribution of the process $\widetilde{Y}_t$
and by $\mathbb{P}_0$ the distribution of the process
$\widetilde{Y}_t^0 := \sigma M_t^H = \sigma b_0 \int_0^ts^{1/2-2H}dB_s$.
The measure $\mathbb{P}_\theta$ is absolutely continuous with respect to $\mathbb{P}_0$ and the Radon-Nikodym
derivative, or the likelihood ratio,  has the following form
(see for instance \cite{LiptserShiryayev}, Theorem 7.19 or apply classical Girsanov Theorem for martingales)
\begin{align*}
\frac{d\mathbb{P}_\theta}{d\mathbb{P}_0} (\widetilde{Y}_t)= &
\exp\left(  - \int_0^t\frac{\theta(2-2H)b_1s^{1-2H} - \sigma^2Hb_2}{\sigma^2b_0^2 s^{1-2H}}\,d\widetilde{Y}_s +
\frac12\int_0^t \frac{[\theta(2-2H)b_1s^{1-2H} - \sigma^2Hb_2]^2}{\sigma^2b_0^2s^{1-2H}}\, ds   \right)\, .
\end{align*}
The MLE is obtained by maximizing the log-likelihood ratio with respect to $\theta$.
Since
$$
\frac{\partial}{\partial \theta} \ln \frac{d\mathbb{P}_\theta }{d\mathbb{P}_0} (\widetilde{Y}_t)=
-\frac{(2-2H)b_1}{\sigma^2 b_0}\widetilde{Y}_t + \theta \frac{(2-2H)b_1^2 t^{2-2H}}{\sigma^2 b_0^2} -
\frac{(2-2H)b_1 H b_2 t}{b_0^2} \, ,
$$
the MLE for parameter $\theta$ has the form
\begin{equation}
\widehat{\theta}_t = \frac{\widetilde{Y}_t}{b_1 t^{2-2H}} + \frac{\sigma^2 H b_2}{b_1t^{1-2H}}.
\end{equation}
\begin{proposition}\label{th:MLEforGeometricalfBM}
The estimate $\widehat{\theta}_t,\ t>0$, is an unbiased estimate for parameter $\theta_0$;
$\lim\limits_{t\to\infty}\widehat{\theta}_t = \theta_0$ with probability one, i.e.
$\widehat{\theta}_t$ is a strong consistent estimate of $\theta_0$;
$t^{1-H}(\widehat{\theta}_t-\theta_0)$ converges in distribution to a Gaussian random variable with zero mean
and variance $\sigma^2/b_1^2$.
\end{proposition}
\begin{proof}
Using the definition of the process $\widetilde{Y}_t$, we represent the estimate $\widehat{\theta}_t$  as follows
\begin{equation}\label{eq:MLEGeometricfBM2}
\widehat{\theta}_t = \theta_0 + \frac{\sigma M_t^H}{b_1t^{2-2H}},
\end{equation}
where $\theta_0$ is the true parameter.

The unbiasedness and asymptotic normality follows immediately from \eqref{eq:MLEGeometricfBM2} and the fact
that $M_t^H$ is a Gaussian random variables with zero mean and variance $t^{2-2H}$.
Since $M_t^H$ is a square integrable  martingale with
unbounded quadratic characteristic $t^{2-2H}\to\infty$, as $t\to\infty$ a.s.,  %, for any $H\in(0,1)$,
by Law of Large Numbers for
Martingales \cite{LiptserShiryayevBookMartingales}, Theorem 2.6.10, $M_t^H/\langle M^H\rangle_t\to 0$ a.s., and hence
consistency follows.
\end{proof}

Note that, in particular, for $H=1/2$ we have $b_1=b_2=1$, and  we recover the classical estimate for the
drift coefficient of geometric Brownian Motion
$$
\widehat{\theta}_t =\frac{Y_t}{t} + \frac{\sigma^2}{2} = \frac{\ln(X_t/X_0)}{t} + \frac{\sigma^2}{2}=
\theta_0 + \frac{\sigma W_t}{t}, \quad t>0\, ,
$$
and its corresponding asymptotic behavior.

\section{Maximum Likelihood Estimator for SPDEs}\label{section:MLEforSPDE}
Consider the diagonalizable equation
\begin{equation}\label{eq:mainMLE}
du(t) = (\mathcal{A}_0+\mathcal{A}_1)u(t)dt + \mathcal{M}u(t)dW^H(t),
\end{equation}
with solution $u(t)=\sum_{k\geq 1}u_k(t)h_k$ given by \eqref{eq:solution2}.
As mentioned in Introduction, if $u$ is observable, then all its Fourier coefficients $u_k$ can be computed.
Thus, we assume that the processes $u_1(t), \ldots, u_N(t)$ can be observed for all
$t\in[0,T]$ and the problem is to estimate the parameter $\theta$ based on this observations. Also, we assume that the
Hurst parameter $H\in(0,1)$ is known for now.

By Definition  \ref{def:solution} of the solution of equation \eqref{eq:mainMLE}
the Fourier coefficients $u_k, \ k\in\mathbb{N}$, have the following dynamics
\begin{equation}\label{eq:DynamicFourierCoeff}
du_k(t) = \alpha_k(\theta)u_k(t)dt + \mu_k u_k(t)dW^H(t), \quad t\in[0,T],
\end{equation}
where $\alpha_k(\theta) = \rho_k + \theta\nu_k, \ k\in\mathbb{N}$.

For every non-zero $u_k(0), \ k\in\mathbb{N}$,
denote by $v_k(t) = \ln (u_k(t)/u_k(0))$, and $\widetilde{v}_k(t) = \int_0^tl(t,s)dv_k(s)$, where
$l(\cdot,\cdot)$ is defined in \eqref{eq:notations}.
By results of Section \ref{subsection:estimationGeometricalFBM} it follows that
there exists a Maximum Likelihood Estimate for $\alpha_k(\theta)$ and it has the form
\begin{equation}
\widehat{\alpha_k(\theta)} = \frac{\widetilde{v}_k(t)}{b_1t^{2-2H}} + \frac{Hb_2\mu_k^2}{b_1t^{1-2H}},
\quad k \geq 1.
\end{equation}
Since $\alpha_k(\theta)$ is a strictly monotone  function in $\theta$, by invariant principle
of MLE under invertible transformations, we can find an MLE for the parameter $\theta$
\begin{equation}\label{eq:MLE}
\widehat{\theta}_{k,t} = \frac{\widetilde{v}_k(t) }{\nu_k b_1t^{2-2H}} + \frac{Hb_2\mu_k^2}{\nu_kb_1t^{1-2H}} -
\frac{\rho_k}{\nu_k}, \quad k\geq 1, \ t\in[0,T] .
\end{equation}

Using the definition of the process $\widetilde{v}_k$, the estimate $\widehat{\theta}_{k,T}$
can be represented as follows
\begin{equation}
\widehat{\theta}_{k,T} = \theta_0 + \frac{\mu_k M_T^H}{b_1\nu_kT^{2-2H}} \, ,
\end{equation}
and by similar arguments to the proof of Proposition \ref{th:MLEforGeometricalfBM}, we have the following result.

\begin{theorem}\label{th:MLE}
Assume that equation \eqref{eq:mainMLE} is diagonalizable and parabolic in the triple
$(\mathbf{H}^{\gamma+m}, \mathbf{H}^\gamma, \mathbf{H}^{\gamma-m})$ for some
$\gamma\in\mathbb{R}, \ m>0$ and $u_0\in\mathbf{H}^\gamma$. Then,
\begin{enumerate}
\item For every $k\geq 1$ and $T>0$, $\widehat{\theta}_{k,T}$ is an unbiased estimator of $\theta_0$.
\item For every fixed $k\geq 1$, as $T\to\infty$, the estimator $\widehat{\theta}_{k,T}$
converges to $\theta_0$ with probability one and $T^{1-H}(\widehat{\theta}_{k,T}-\theta_0)$
converges in distribution to a Gaussian random variable with zero mean and variance $\mu_k^2/b_1^2\nu_k^2$.
\item If, in addition,
\begin{equation}\label{eq:CondConvMLE}
\lim\limits_{k\to\infty}\left|\frac{\mu_k}{\nu_k}\right| = 0 ,
\end{equation}
then for every fixed $T>0$, $\lim\limits_{k\to\infty}\widehat{\theta}_{k,T}=\theta_0$ with
probability one, and $|\nu_k/\mu_k|(\widehat{\theta}_{k,T}-\theta_0)$ converges in
distribution to a Gaussian random variable with zero mean and variance $T^{2H-2}/b_1^2$.
\end{enumerate}
\end{theorem}

\begin{remark}\label{remark:OrderAndEstimation}
The parabolicity conditions \eqref{eq:parabolic1}-\eqref{eq:parabolic2} and
MLE consistency condition \eqref{eq:CondConvMLE} in general are not connected.
In terms of operator's order, parabolicity states that the order of operator $\mathcal{M}$
from the diffusion term is smaller than half of the order of the operators
$\mathcal{A}_0$ and $\mathcal{A}_1$ from deterministic part.
Condition \eqref{eq:CondConvMLE}, that guarantees the consistency of MLE as number of Fourier
coefficients increases, assumes that the order of operator $\mathcal{M}$ from the diffusion
part does not exceed the order of the operator $\mathcal{A}_1$ from deterministic part that
contains the parameter of interest $\theta$.
\end{remark}

By Theorem \ref{th:MLE} it follows that the consistency and asymptotic normality
of the estimates $\widehat{\theta}_{k,T}$ can be achieved  in two ways: by increasing time $T$ or by
increasing the number of Fourier coefficients $k$. In both cases the quality of the estimate
is improved by decreasing its variance.

It is interesting to note that
$\mathrm{Var}\left( \widehat{\theta}_{k,T} - \theta_0 \right) = \mu_k^2T^{2H-2}/b_1^2\nu_k^2$
also depends on Hurst parameter $H$.  For $H>1/2$ the constant
$1/b_1$ is close to one, and increases as function of $H$ for $H\in (0,1/2)$. % (see Figure \ref{fig:b1}).
The function $t^{2H-2}$ increases in $H$ for any $t>1$. The constants $\mu_k$ and $\nu_k, \ k\geq 1$, do not
depend on $H$.
Overall, $T^{2H-2}/b_1^2$ increases in $H$ for any $t>1$
%(see Figure \ref{fig:normalityVarianceGeofBMFixedT} and Figure \ref{fig:normalityVarianceGeofBM}),
and thus quality of the estimates is higher for smaller $H$. \\

%

%\begin{center}
%\begin{figure}[ht]
%\begin{minipage}[t]{0.5\linewidth}
%\centering
%\includegraphics[width=.8\linewidth]{b1}
%\caption{ {\footnotesize $1/b_1^2$ as function of Hurst parameter $H$}}   \label{fig:b1}
%\end{minipage}\hspace{.5cm}
%\begin{minipage}[t]{0.5\linewidth}
%\centering
%\includegraphics[width=.8\linewidth]{variance_geofBMFixedT}
%\caption{  {\footnotesize  $\frac{T^{2H-2}}{b_1^2}$ as function of $H$ for fixed $T=1.5$}}   \label{fig:normalityVarianceGeofBMFixedT}
%\end{minipage}
%\end{figure}
%\end{center}
%
%
%\begin{center}
%\begin{figure}[h]
%\centering
%\includegraphics[width=.6\linewidth]{variance_geofBM}
%\caption{  {\footnotesize  Behavior of
%$\mathrm{Var}\left( \widehat{\theta}_{k,T} - \theta_0 \right) =
%\mu_k^2T^{2H-2}/b_1^2\nu_k^2, \ H\in(0,1), \ t>1 $}}   \label{fig:normalityVarianceGeofBM}
%\end{figure}
%\end{center}

As mentioned before, due to the fact that the
probability measures generated by the solution $u$ of the original SPDE are singular,
it is possible to estimate $\theta$ exactly on any finite interval of time $[0,T]$. A natural question
is wether we can improve the quality of the estimates by considering several Fourier coefficients
$u_k(t)$. The answer is that by statistical methods used above this is not possible. First,
note that the measures associated to any two or more processes $u_k$ are singular, and
thus MLE does not exist for such vector-valued functions. In other words, by considering
two or more Fourier coefficients $u_k$, we get a singular model, a fact that
will be explored in the next section. Also, since each process $u_k$ is driven by the same
noise, each individual Fourier coefficient $u_k$ contains the same amount of information:
the sigma-algebra generated by $u_k(t), \ t\in[0,T]$ coincides with the sigma-algebra
generated by $W^{H}(t), \ t\in[0,T]$. However, the speed of convergence of the sequence
$\widehat{\theta}_{k,T}$ can be improved by using accelerating convergence technics from numerical analysis.
Two methods have been discussed into details in \cite{IgorSergey2007}: the  weighted  average method and
Aitken's $\Delta^2$ method. For sake of completeness, we will state here the corresponding results
applied to the sequence $\{\widehat{\theta}_{k,T}\}_{k\geq 1}$.

\noindent {\bf Weighted averaging.} Suppose that $\beta_k, \ k\geq 1$, is a sequence of
non-negative numbers such that $\sum_{k\geq 1}\beta_k=+\infty$, and consider
the weighted averaging estimator
\begin{equation}\label{eq:WeightedMLE}
\widehat{\theta}_{(N,T)} =\sum\limits_{k=1}^N \beta_k\widehat{\theta}_{k,T}\Big/\sum\limits_{k=1}^N\beta_k
\quad N\geq 1, \ T>0\, .
\end{equation}
Then (a) $\widehat{\theta}_{(N,T)}$ is an unbiased estimator of $\theta_0$ for every $N\geq 1$ and $T>0$;
(b) $\lim\limits_{T\to\infty}\widehat{\theta}_{(N,T)}=\theta_0$ a.e. for every $N\geq 1$ (consistency in $T$);
(c) if in addition the consistency condition
 \eqref{eq:CondConvMLE} is fulfilled, then   $\lim\limits_{N\to\infty}\widehat{\theta}_{(N,T)}=\theta_0$ with probability one for every $T>0$ (consistency in $N$).

\bigskip

\noindent {\bf Aitken's $\Delta^2$ method.} Define the following sequence of
estimates
\begin{equation}\label{eq:Aitken}
\widetilde{\theta}_k = \widehat{\theta}_{k,T} - \frac{(\widehat{\theta}_{k+1, T} -\widehat{\theta}_{k,T})^2}%
{\widehat{\theta}_{k+2,T}+2\widehat{\theta}_{k+1,T}-\widehat{\theta}_{k,T}} .
\end{equation}
One can show that the new sequence $\widetilde{\theta}_k$ converges to the true parameter
$\theta_0$ with probability one. Moreover, if $\mu_k/\nu_k\sim \alpha k^{-\delta}$ for
some $\alpha,\delta>0$, then
$$
\frac{\mathbb{E}(\widetilde{\theta}_k-\theta_0)^2}{\widehat{(\theta}_{k,T}-\theta_0)^2} \sim \frac{1}{(1+\delta_1)^2},
$$
and if $\mu_k/\nu_k = (-1)^k/k$, then
$$
\frac{\mathbb{E}(\widetilde{\theta}_k-\theta_0)^2}{\widehat{(\theta}_{k,T}-\theta_0)^2} \sim \frac{c}{k^2},\quad c>0.
$$
In both cases, the new sequence $\widetilde{\theta}_k$ converges faster than $\widehat{\theta}_{k,T}$ to
$\theta_0$.

The proofs of the above results follows from Theorem \ref{th:MLE} and some direct computations, which will be omitted here.

%\newpage
%Hurst parameter $H\in(0,1)$. Denote by $h_k$ and $\alpha_k$ the eigenvectors and corresponding
%eigenvalues of the Laplace operator $\Delta u = u_{xx}$ on $[0,\pi]$ with zero boundary conditions.
%Then $h_k(x) = \sqrt{2/\pi}\sin(kx), \ \alpha_k=-k^2, \ k\geq 1$. Let
%$$
%u_k(t) = \int_0^\pi u(t,x) h_k(x)dx, \quad k\geq 1,
%$$
%be the Fourier coefficients of the solution $u$, then
%$$
%u(t,x) = \sum_{k\geq 1} u_k(t) h_k(x)
%$$
%and each Fourier coefficient satisfies the following stochastic differential
%equation
%\begin{equation}\label{eq:FourierCoefficient}
%u_k(t) = u_k(0) + \theta \alpha_k \int_0^t  u_k(s)ds + \sigma \int_0^t u_k(s)dW^H(s) \ .
%\end{equation}
%For $H>1/2$ the equation \eqref{eq:FourierCoefficient} can be interpreted as
%Lebesgue-Stieltjes integral, for $H=1/2$ we regard it as It\^o integral and
%for $H<1/2$ we interpreted this equation as ....

\section{Exact Estimates}\label{section:ExactEstimates}
In regular models the unknown parameter can be found only approximatively, and
the consistency is gained  either in large sample or small noise regime. For singular models
the parameter  can be found exactly. For example, if all Fourier coefficients of the solution
$u$ of equation \eqref{eq:main} are known, according to the results from previous sections, one can
find the value of  $\theta_0$ exactly, on any interval of time $[0,T]$.
The possibility to evaluate $\theta_0$ exactly is based on
singularity of the measures generated by $u^{\theta}$ for different values of $\theta$. However,
while theoretically it is possible to estimate the true parameter exactly,
in practice we (or computer) can perform only a finite number of operations. Recall that the measures associated
to an individual Fourier coefficient $u_k^\theta$ are regular, while a vector consisting
of any two or more Fourier coefficients will produce measures that are singular. In this section
we will explore this singularity, and show that in fact the true parameter can be estimated exactly from
a finite number of  Fourier coefficients. Moreover, the described method allow to find both parameters $\theta$ and $H$, either individually or simultaneously.

Following \cite{IgorSergey2007} we say that an estimator is {\it closed-form exact} if
it produces the exact value of the parameter of interest after finite
number of additions, substraction, multiplications, and divisions on the elementary functions of the observations.

Closed-form  exact estimates exists for the model \eqref{eq:main} if we assume that
observations are $u_k(t), k\geq 1, \ t\in[0,T]$.
For every non-zero Fourier coefficient $u_k$ of the form \eqref{eq:DynamicFourierCoeff},
set $v_k(t)=\ln u_k(t)/u_k(0), \ t\in[0,T]$. Then
\begin{equation}\label{eq:v_k}
v_k(t) = (\rho_k + \theta\nu_k) t - \frac12 \mu_k^2 t^{2H} + \mu_kW^H(t) \ .
\end{equation}

\bigskip
\noindent {\bf Case 1.  $\theta$ unknown, $H$ known.}
Assume that $\nu_k\mu_m\neq\nu_m\mu_k$ for some $k,m\in\mathbb{N}$. Then, taking \eqref{eq:v_k}
for these $k$ and $m$, by direct arithmetic evaluations, one gets the
exact estimate of the parameter $\theta$
\begin{equation}\label{eq:exactEstimate}
\theta = \frac{\mu_m v_k- \mu_kv_m + (\rho_m\mu_k-\rho_k\mu_m)t + \frac{1}{2} (\mu_k^2\mu_m-\mu_m^2\mu_k) t^{2H}}
{t(\nu_k\mu_m-\nu_m\mu_k)} \ ,
\end{equation}
for any $t>0$ and $k,m\in\mathbb{N}$ for which $\nu_k\mu_m\neq\nu_m\mu_k$.

Note that if $\mu_k=\mu_m$ then the above exact estimate does not depend on $H$, and
$\theta$ can be evaluated even if $H$ is unknown.  This is the case, for example, if
$\mathcal{M}$ is the identity operator (see Example 1 below).

\bigskip

\noindent {\bf Case 2. $H$ unknown, $\theta$ known. }
Assume now that the parameter of interest is the Hurst parameter $H$ and assume that $\theta$ is known.
By the same arguments as above, one can solve for $H$ the system of two equations generated by
\eqref{eq:v_k} for some $k$ and $m$, and get the following exact estimate for $H$
\begin{equation}
H=\frac{1}{2 \ln t} \ln\left[
\frac{(\rho_k+\theta\nu_k)t\mu_m - (\rho_m+\theta\nu_m)t\mu_k - v_k\mu_m + v_m\mu_k}
{2\mu_k\mu_m(\mu_k-\mu_m)} \right] \ ,
\end{equation}
for any $t>0$, $k\neq m$, and under assumption that the expression under logarithm is positive and finite.

\bigskip

\noindent {\bf Case 3. Both $\theta$ and $H$ unknown.}
Denote by $\alpha_{k,m}:=(\nu_k\mu_m - \nu_m\mu_k)t,  \
\beta_{k,m}:=1/2(\mu_m^2\mu_k - \mu_k^2\mu_m)$ and
$\delta_{k,m} := v_k\mu_m - v_m\mu_k -\rho_k\mu_m t -\rho_m\mu_k t$.
Assume that for some $m,k,i,j\in\mathbb{N}$, $\alpha_{k,m}\beta_{i,j}\neq \alpha_{k,m}\beta_{i,j}$. Then
the following exact estimate for $\theta$ holds true
\begin{equation}\label{eq:ExactEstimateThetaH1}
\theta = \frac{\delta_{k,m}\beta_{i,j} - \delta_{i,j}\beta_{k,m}}
{\alpha_{k,m}\beta_{i,j}-\alpha_{i,j}\beta_{k,m}} \ .
\end{equation}
If in addition $\delta_{k,m}\alpha_{i,j}\neq\delta_{i,j}\alpha_{k,m}$, then
there exists an exact estimate for Hurst parameter $H$ given by
\begin{equation}\label{eq:ExactEstimateThetaH2}
H = \frac12 \log_t \frac{\delta_{k,m}\alpha_{i,j} - \delta_{i,j}\alpha_{k,m}}
{\beta_{k,m}\alpha_{i,j} - \beta_{i,j} \alpha_{k,m}} \ .
\end{equation}
Note that for this case, generally speaking, it is sufficient to know only three Fourier
coefficients, i.e. some of the indices $k,m,i,j$ can coincide.

\begin{remark} $\quad$
\begin{itemize}
\item[(a)]Applying the above idea, closed-form exact estimates can be obtained for equations driven by
several fractional Brownian motions, even with different Hurst parameters. If we assume that
the noise is driven by $n$ fBMs, and that one of the parameters $\theta$ or $H$ is known,
then by considering $n+1$ Fourier coefficients we can eliminate all
noises and get a closed-form estimate as a solution, under some non-degeneracy assumptions.
Respectively, if both parameters are unknown, then one can estimate them by considering
$n+2$ Fourier coefficients.
\item[(b)]Note that the construction of the exact estimates assumed only the existence of the solution and did not impose any additional assumptions on the
order of the operators $\mathcal{A}_0, \ \mathcal{A}_1, \ \mathcal{M}$, in contrast to
MLE estimates where the consistency holds only under additional assumptions on order of corresponding operators.
\item[(c)] The MLE $\widehat{\theta}_{k,T}$ depend on the whole trajectory of the Fourier coefficient $u_k(t), \ t\in[0,T]$.
All exact estimates depend only on
initial and terminal value of $u_k$'s.
\end{itemize}
\end{remark}

\section{Examples}
We conclude the paper with two practical examples where we explore some of the estimates proposed above.

{\bf Example 1.} {\it Stochastic heat equation.} Let $\theta$ be a positive number, and consider the following equation
\begin{equation}\label{eq:heatEquation}
du(t,x) = \theta u_{xx}(t,x)dt + u(t,x)dW^H(t), \quad t>0, \ x\in(0,1),
\end{equation}
with zero boundary conditions and some nonzero initial value $u(0)\in L_2(0,1)$. In this case
the operator $\mathcal{A}_1$ is the Laplace operator on $(0,1)$ with zero boundary conditions that
has the eigenfunctions $h_k(x)=\sqrt{2/\pi}\sin(kx), \ k>0$, and eigenvalues
$\nu_k=-k^2,\ \rho_k=0,\ \mu_k=1, \ k>0$.  Assume that $u(t,x)$ is known for $x\in[0,1]$ and $t\in[0,T]$, hence
$u_k(t):=\int_0^1h_k(x)u(t,x)dx, \ k\in\mathbb{N}$, is observable. Denote by
$v_k(t):=\log( u_k(t)/u_k(0))$  for every $k\in\mathbb{N}$, and $u_k(0)\neq 0$.
By Theorem \ref{th:MLE}, the MLE for $\theta$ has the form
$$
\widehat{\theta}_k = - \frac{\int_0^Tl(T,s)dv_k(s)}{k^2 b_1 T^{2-2H}}
- \frac{Hb_2}{k^2 b_1 T^{1-2H}} , \ k\in\mathbb{N}.
$$

The exact estimates \eqref{eq:exactEstimate} for $\theta$ are given by
$$
\theta = \frac{1}{T(m^2-k^2)} \ln\frac{u_k(T)u_m(0)}{u_m(T)u_k(0)} \, ,
$$
for any $k \neq m$ and $T>0$. Note that the exact estimates do not depend on $H$.
However there are no exact-type estimates for $H$.
%{\tt [include figure or a table here, maybe]}

\noindent  {\bf Example 2.} Assume that $G$ is a bounded domain in $\mathbb{R}^d$, and let
$\Delta$ be the Laplace operator on $G$ with zero boundary conditions. Then $\Delta$ has
only point spectrum with countable many eigenvalues, call them $\sigma_k, k\in\mathbb{N}$.
Moreover, the set of corresponding eigenvalues forms an orthonormal basis in
$L_2(G)$; the eigenvalues can be arranged so that $0<-\sigma_1\leq -\sigma_2\leq \ldots$;
the eigenvalues have the asymptotic $\sigma_k\sim k^{2/d}$.   In the space $\mathbf{H}^0(G)$
let us consider the following stochastic evolution equation
$$
du(t)= [\Delta u(t) + \theta u(t)]dt + (1-\Delta)^r u(t) dW^H(t) \ ,
$$
with some nonzero initial values in $\mathbf{H}^0(G)$, and some $r\in\mathbb{R}$.
According to our notations we have the operators
$\mathcal{A}_0=\Delta, \ \mathcal{A}_1=I, \ \mathcal{M}=(1-\Delta)^r$, with corresponding eigenvalues
$\nu_k = 1, \ \rho_k=\sigma_k, \ \mu_k=(1+\sigma_k)^r $.  The equation is diagonalizable, and by
Theorem \ref{th:existance}, it has a unique solution in the triple
$(\mathbf{H}^{1}, \mathbf{H}^0,\mathbf{H}^{-1})$ for any $r\leq 1/2$.

The maximum likelihood estimate in this case has the form
$$
\widehat{\theta}_{N,t} = \frac{\widetilde{v}_k(t)}{b_1\sigma_k t^{2-2H}} +
\frac{Hb_2(1-\sigma_k)^{2r}}{\sigma_kb_1t^{1-2H}} - \frac{1}{\sigma_k}, \ t>0, \ k\in\mathbb{N},
$$
which is an unbiased estimate of the parameter $\theta$.

\noindent {\bf (2a)} {\it Large time asymptotics.} $\lim\limits_{t\to\infty}\widehat{\theta}_{k,t} = \theta_0$ a.s. for all $k\geq 1$;
$\lim\limits_{t\to\infty} t^{1-H}(\widehat{\theta}_{k,t} -\theta_0  ) \overset{d}{=} \xi$, where $\xi\sim \mathcal{N}(0, (1-\sigma_k)^2/b_1^2)$.

\noindent {\bf (2b)} {\it Consistency in number of spatial Fourier coefficients.}  Assume that $r<0$. Then $\lim\limits_{k\to\infty}\widehat{\theta}_{k,t} = \theta_0$ a.s., for every $t>0$, and
 the sequence $(1-\sigma_k)^{-1}(\widehat{\theta}_{k,t} - \theta_0)$ converges in distribution to a Gaussian random variable with mean zero and variance $t^{2H-2}/b_1^2$.
 If $r\in[0,1/2]$ the solution still exists in the space $\mathbf{H}^0(G)$, while the estimate $\widehat{\theta}_{k,t}$ is not consistent in $k$.

\noindent {\bf (2b)} {\it Exact estimates.} Let $v_k(t)=\ln(u_k(t)/u_k(0))$. Assume that Hurst parameter $H$ is known. Then we have the following exact estimate for $\theta$
\begin{align*}
\theta =&
\frac{ (1-\sigma_m)^r v_k - (1-\sigma_k)^r v_m}{ t((1-\sigma_m)^r - (1-\sigma_k)^r)} +
     \frac{\sigma_m(1-\sigma_k)^r - \sigma_k (1-\sigma_m)^r}{ (1-\sigma_m)^r - (1-\sigma_k)^r }  \\
     & + \frac{t^{2H-1}}{2}\frac{(1-\sigma_k)^{2r}(1-\sigma_m)^r - (1-\sigma_m)^{2r}(1-\sigma_k )^r}{ (1-\sigma_m)^r - (1-\sigma_k)^r },
\end{align*}
for any $k\neq m$ and $t>0$.

If $\theta$ is known, then the Hurst parameter $H$ can be found by
$$
H = \frac{1}{2}\log_t   \frac{ [(\sigma_k + \theta)(1-\sigma_m)^r - (\sigma_m+\theta) (1-\sigma_k)^r]t - v_k(1-\sigma_k)^r + v_m(1-\sigma_m)^r }
{2 (1-\sigma_k)^{2r}( 1- \sigma_m)^r - (1-\sigma_m)^{2r}(1-\sigma_k)^r } ,
$$
for any $k\neq m, \ t>0$.

Finally one can write the exact estimates \eqref{eq:ExactEstimateThetaH1} and \eqref{eq:ExactEstimateThetaH2} for the case when both parameters $\theta$ and $H$ are unknown.
Note that the exact estimates exists for all $r$ as long as the solution exists (maybe in a larger space) and the Fourier coefficients $u_k(t)$ are computable.

%\section{Conclusion}
%We proposed two classes of estimates for the drift coefficient of general parabolic stochastic differential equations under assumption
%that the solution is observed continuously in time and space.
%All estimates are expressed in terms of the Fourier coefficients of the solution.
%First class of estimates (see Section \ref{section:MLEforSPDE})
%is based on traditional maximum likelihood estimates.

\bibliographystyle{amsplain}
%\bibliography{D:/_research/latex/lib_igor/igor_bib_probability}
%\bibliography{C:/Documents and Settings/Administrator/Desktop/working_files/_research/latex/lib_igor/igor_bib_probability}

%\bibliography{igor_bib_probability}

\def\cprime{$'$} \def\cprime{$'$}
\providecommand{\bysame}{\leavevmode\hbox to3em{\hrulefill}\thinspace}
\providecommand{\MR}{\relax\ifhmode\unskip\space\fi MR }
% \MRhref is called by the amsart/book/proc definition of \MR.
\providecommand{\MRhref}[2]{%
  \href{http://www.ams.org/mathscinet-getitem?mr=#1}{#2}
}
\providecommand{\href}[2]{#2}

\end{document}